\documentclass[11pt]{amsart}

\allowdisplaybreaks[4]
\newtheorem{proposition}{Proposition}[section]
\newtheorem{lemma}[proposition]{Lemma}
\newtheorem{corollary}[proposition]{Corollary}
\newtheorem{theorem}[proposition]{Theorem}

\theoremstyle{definition}

\newtheorem{example}[proposition]{Example}

\newtheorem{remarks}[proposition]{Remarks}

\newcommand{\thlabel}[1]{\label{th:#1}}
\newcommand{\thref}[1]{Theorem~\ref{th:#1}}
\newcommand{\selabel}[1]{\label{se:#1}}

\newcommand{\lelabel}[1]{\label{le:#1}}
\newcommand{\leref}[1]{Lemma~\ref{le:#1}}
\newcommand{\prlabel}[1]{\label{pr:#1}}
\newcommand{\prref}[1]{Proposition~\ref{pr:#1}}
\newcommand{\colabel}[1]{\label{co:#1}}

\newcommand{\relabel}[1]{\label{re:#1}}

\newcommand{\exlabel}[1]{\label{ex:#1}}
\newcommand{\exref}[1]{Example~\ref{ex:#1}}

\newcommand{\eqlabel}[1]{\label{eq:#1}}
\newcommand{\equref}[1]{(\ref{eq:#1})}

\newcommand{\Hom}{{\rm Hom}}
\newcommand{\End}{{\rm End}}

\newcommand{\im}{{\rm Im}\,}

\def\ot{\otimes}

\def\ZZ{{\mathbb Z}}
\def\QQ{{\mathbb Q}}

\newcommand{\Cc}{\mathcal{C}}
\newcommand{\Dd}{\mathcal{D}}

\newcommand{\Ff}{\mathcal{F}}

\newcommand{\Mm}{\mathcal{M}}
\newcommand{\Pp}{\mathcal{P}}

\def\*C{{}^*\hspace*{-1pt}{\Cc}}
\def\text#1{{\rm {\rm #1}}}

\def\ul{\underline}

\def\equal#1{\smash{\mathop{=}\limits^{#1}}}

\usepackage{amssymb}

\usepackage{color,amssymb,graphicx,amscd}
\usepackage{xypic}

\begin{document}

\title[Braidings on the category of bimodules]
{Braidings on the category of bimodules, Azumaya algebras and epimorphisms of rings}

\author{A. L. Agore}
\address{Faculty of Engineering, Vrije Universiteit Brussel, Pleinlaan 2, B-1050 Brussels, Belgium}
\email{ana.agore@vub.ac.be and ana.agore@gmail.com}

\author{S. Caenepeel}
\address{Faculty of Engineering, Vrije Universiteit Brussel, Pleinlaan 2, B-1050 Brussels, Belgium}
\email{scaenepe@vub.ac.be}

\author{G. Militaru}
\address{Faculty of Mathematics and Computer Science, University of Bucharest, Str.
Academiei 14, RO-010014 Bucharest 1, Romania}
\email{gigel.militaru@fmi.unibuc.ro and gigel.militaru@gmail.com}
\subjclass[2010]{16T10, 16T05, 16S40}

\keywords{braided category, epimorphism of rings, Azumaya algebra,
separable functor, quantum Yang-Baxter equation.}

\thanks{A.L. Agore is Aspirant fellow of FWO-Vlaanderen. S.Caenepeel is supported by
FWO project G.0117.10 ``Equivariant Brauer groups and Galois
deformations''. A.L. Agore and G. Militaru are supported by the
CNCS - UEFISCDI grant no. 88/05.10.2011 ``Hopf algebras and
related topics''.}

\begin{abstract}
Let $A$ be an algebra over a commutative ring $k$. We prove that
braidings on the category of $A$-bimodules are in bijective
correspondence to canonical R-matrices, these are elements in
$A\ot A\ot A$ satisfying certain axioms. We show that all
braidings are symmetries. If $A$ is commutative, then there exists
a braiding on ${}_A\Mm_A$ if and only if $k\to A$ is an
epimorphism in the category of rings, and then the corresponding
$R$-matrix is trivial. If the invariants functor $G = (-)^A:\
{}_A\Mm_A\to \Mm_k$ is separable, then $A$ admits a canonical
R-matrix; in particular, any Azumaya algebra admits a canonical
R-matrix. Working over a field, we find a remarkable new
characterization of central simple algebras: these are precisely
the finite dimensional algebras that admit a canonical R-matrix.
Canonical R-matrices give rise to a new class of examples of
simultaneous solutions for the quantum Yang-Baxter equation and
the braid equation.
\end{abstract}

\maketitle

\section*{Introduction}
Braided monoidal categories play a key role in several areas of
mathematics like quantum groups, noncommutative geometry, knot
theory, quantum field theory and 3-manifolds. It is well-known
that the category ${}_A\Mm_A$ of bimodules over an algebra $A$
over a commutative ring $k$ is monoidal. The aim of this paper is
to give an answer to the following natural question: given an
algebra $A$, describe all braidings on ${}_A\Mm_A$.  Besides the
purely categorical significance this problem is also relevant in
noncommutative geometry where braidings on ${}_A\Mm_A$ are used to
develop the theory of wedge products of differential forms or
connections on bimodules. The question is not as obvious as it
seems: a first attempt might be to use the switch map to define
the braiding, but this is not well-defined, even in the case when
$A$ is a commutative algebra. However, there are non-trivial
examples of braidings on the category of bimodules. For example,
let $A=\Mm_n (k)$ be a matrix algebra; then $c_{M,N}:\ M\ot_A N\to
N\ot_AM$ given by the formula
$$c_{M, N} (m\ot_A n) =
\sum_{i,j,t=1}^n e_{ij} \, n \, e_{ti}\ot_A \, m \, e_{jt}
$$
is a braiding on the category of $A$-bimodules (see \exref{2.5}
for full detail). A first general result is \thref{1.1}, stating
that braidings on the category of $A$-bimodules are in bijective
correspondence with canonical R-matrices, these are invertible
elements $R$ in the threefold tensor product $A\ot A\ot A$,
satisfying a list of axioms. In this situation, we will say that
$(A,R)$ is an algebra with a canonical R-matrix. Actually, this
result is inspired by a classical result of Hopf algebras:
braidings on the category of (left) modules over a bialgebra $H$
are in one-to-one correspondence with quasitriangular structures
on $H$, these are elements $R$ in the two-fold tensor product
$H\ot H$ satisfying certain properties. We refer to \cite[Theorem
10.4.2]{Montgomery} for detail. The next step is to reduce the
list of axioms to two equations, a centralizing condition and a
normalizing condition, and then we can prove in  \thref{2.1} that
all braidings on a category of bimodules are symmetries. In the
situation where $A$ is commutative, we have a complete
classification: $A$ admits a canonical R-matrix $R$ if and only if
$k\to A$ is an
epimorphism in the category of rings, and then $R$ is trivial, see \prref{1.2}.\\
The invariants functor $G = (-)^A:\ {}_A\Mm_A\to \Mm_k$ has a left
adjoint $F = A\ot -$. We prove that $G$ is a separable functor
\cite{NastasescuVV89,Rafael} if and only if $G$ is fully faithful
and this implies that $A$ admits a canonical R-matrix. The converse property
also holds if $A$ is free as a $k$-module, and then the braiding on the category
of $A$-bimodules is unique, cf. \thref{2.2}.\\
Azumaya algebras were introduced in \cite{AG} under the
name central separable algebras; a more restrictive class was
considered earlier by Azumaya in \cite{Az}. Azumaya algebras are
the proper generalization of central simple algebras to
commutative rings. The Brauer group consists of the set of Morita
equivalence classes of Azumaya algebras. There exists a large
literature on Azumaya algebras and the Brauer group, see for
example the reference list in \cite{stef}. $A$ is an Azumaya
algebra if and only if $G$ is an equivalence of categories, and
then $G$ is separable. Therefore the category of bimodules over an
Azumaya algebra is braided monoidal, that is any Azumaya algebra
admits a canonical R-matrix. $R$ can be described explicitly in the cases
where $A$ is a matrix algebra or a quaternion algebra, see Examples
\ref{ex:2.5} and \ref{ex:2.6}. Not every algebra with a canonical R-matrix
is Azumaya; for example $\QQ$ is not a $\ZZ$-Azumaya algebra, but
$1\ot 1\ot 1$ is a canonical R-matrix, since $\ZZ\to \QQ$ is an epimorphism of rings.
Thus algebras with a canonical R-matrix can be viewed as generalizations of
Azumaya algebras.\\
Applying \thref{2.2} to finite dimensional algebras over fields, we obtain a new
characterization of central simple algebras, namely central simple algebras are
the finite dimensional algebras admitting a canonical R-matrix. As a final application,
we construct a simultaneous solution of the quantum
Yang-Baxter equation and the braid equation from any canonical R-matrix, see
\thref{qybe}.

\section{Preliminary}\selabel{0}
\subsection*{Azumaya algebras}\selabel{0.2}
Let $k$ be a commutative ring and $A$ a $k$-algebra.
Unadorned $\ot$ means $\ot_k$. ${}_A\Mm_A$ is
the $k$-linear category of $A$-bimodules.
It is well-known that we
have a pair of adjoint functors $(F, G)$ between the category of
$k$-modules $\Mm_k$ and the category of $A$-bimodules ${}_A\Mm_A$.
For a $k$-module $N$, $F(N)=A\ot N$, with $A$-bimodule structure
$a(b\ot n)c=abc\ot n$, for all $a$, $b$, $c\in A$ and $n\in N$.
For an $A$-bimodule $M$, $G(M) = M^A = \{ m\in M \, | \, am = ma,
\forall a \in A \} \cong {}_A\Hom_A (A, M)$. The unit $\eta$ and
the counit $\varepsilon$ of the adjoint pair $(F, G)$ are given by
the formulas
 $$\begin{array}{ccc}
 \eta_N:\ N\to (A\ot N)^A&;&\eta_N(n)=1\ot n;\\
 \varepsilon_M:\ A\ot M^A\to M&;&\varepsilon_M(a\ot m) = am = ma
 \end{array}$$
for all $n\in N$, $a\in A$ and $m\in M^A$. Recall that $A$ is an
Azumaya algebra if $A$ is faithfully projective as a $k$-module,
that is, $A$ is finitely generated, projective and faithful, and
the algebra map
 \begin{equation}\eqlabel{azumaya}
 F:\ A^e=A\ot A^{\rm op}\to \End_k(A),~~F(a\ot b)(x)=axb
 \end{equation}
is an isomorphism. Azumaya algebras can be characterized in
several ways; perhaps the most natural characterization is the
following: $A$ is an Azumaya algebra if and only if the adjoint
pair $(F,G)$ is a pair of inverse equivalences, see \cite[Theorem
III.5.1]{KO}. Another characterization is that $A$ is central and separable
as a $k$-algebra.

\subsection*{Separable functors}\selabel{0.3}
Recall from \cite{NastasescuVV89} that a covariant functor $F:\
\Cc\to \Dd$ is called separable if the natural transformation
$$\Ff:\
\Hom_\Cc(\bullet,\bullet)\to\Hom_\Dd(F(\bullet),F(\bullet))~;~
\Ff_{C,C'}(f)=F(f)$$ splits, that is, there is a natural
transformation
$$\Pp:\
\Hom_\Dd(F(\bullet),F(\bullet))\to \Hom_\Cc(\bullet,\bullet)$$
such that $\Pp\circ \Ff$ is the identity natural transformation.
Rafael's Theorem \cite{Rafael} states that the left adjoint $F$ in
an adjoint pair of functors $(F,G)$ is separable if and only if
the unit of the adjunction $\eta : 1_{\Cc} \to GF$ splits; the
right adjoint $G$ is separable if and only if the counit
$\varepsilon: FG \to 1_{\Dd}$  cosplits, that is, there exists a
natural transformation $\zeta: 1_{\Dd} \rightarrow FG$ such that
$\varepsilon\circ\zeta$ is the identity natural transformation. A
detailed study of separable functors can be found in \cite{book}.

\subsection*{Braided monoidal categories}\selabel{0.1}
A monoidal category $\Cc=(\Cc,\ot, I,a,l,r)$ consists of a
category $\Cc$, a functor $\ot:\ \Cc\times\Cc\to \Cc$, called the
tensor product, an object $I\in \Cc$ called the unit object, and
natural isomorphisms $a:\ \ot\circ (\ot\times \Cc)\to \ot\circ
(\Cc\times \ot)$ (the associativity constraint), $l:\ \ot\circ
(I\times \Cc)\to \Cc$ (the left unit constraint) and $r:\ \ot\circ
(\Cc\times I)\to \Cc$ (the right unit constraint). $a$, $l$ and
$r$ have to satisfy certain coherence conditions, we refer to
\cite[XI.2]{K} for a detailed discussion. $\Cc$ is called strict
if $a$, $l$ and $r$ are the identities on $\Cc$. McLane's
Coherence Theorem asserts that every monoidal category is monoidal
equivalent to a strict one, see \cite[XI.5]{K}. The categories
that we will consider are not strict, but they can be
treated as if they were strict.\\
Let $\tau:\ \Cc\times \Cc\to\Cc\times \Cc$ be the flip functor. A prebraiding on $\Cc$
is a natural transformation $c:\ \ot\to \ot\circ \tau$ satisfying the following equations,
for all $U,V,W\in \Cc$:
$$c_{U,V\ot W}=(V\ot c_{U,W})\circ (c_{U,V}\ot W)~~;~~
c_{U\ot V,W}=(c_{U,W}\ot V)\circ (U\ot c_{V,W}).$$
$c$ is called a braiding if it is a natural isomorphism.
$c$ is called a symmetry if $c_{U,V}^{-1}=c_{V,U}$, for all $U,V\in \Cc$. We refer to
\cite[XIII.1]{K} for more detail.

\section{Braidings on the category of bimodules}\selabel{1}
Let $A$ be an algebra over a commutative ring $k$ and ${}_A\Mm_A =
({}_A\Mm_A, - \ot_A - , A)$ the monoidal category of
$A$-bimodules. $A^{(n)}$ will be a shorter notation for the
$n$-fold tensor product $A\ot\cdots\ot A$, where $\ot = \ot_k$. An
element $R \in A^{(3)}$ will be denoted by $R = R^1 \ot R^2 \ot
R^3$, where summation is implicitly understood. Our first aim is
to investigate braidings on ${}_A\Mm_A$.

\begin{theorem}\thlabel{1.1}
Let $A$ be a $k$-algebra. Then there is a bijective correspondence
between the class of all braidings $c$ on ${}_A\Mm_A$ and the set
of all invertible elements $R=R^1\ot R^2\ot R^3\in A^{(3)}$
satisfying the following conditions, for all $a\in A$:
\begin{eqnarray}
R^1\ot R^2\ot aR^3&=&R^1a\ot R^2\ot R^3\eqlabel{1.1.2}\\
aR^1\ot R^2\ot R^3&=&R^1\ot R^2a\ot R^3\eqlabel{1.1.3}\\
R^1\ot aR^2\ot R^3&=&R^1\ot R^2\ot R^3a\eqlabel{1.1.4}\\
R^1\ot R^2\ot 1\ot R^3&=&r^1R^1\ot r^2\ot r^3R^2\ot R^3\eqlabel{1.1.5}\\
R^1\ot 1\ot R^2\ot R^3&=&R^1\ot R^2r^1\ot r^2\ot R^3r^3\eqlabel{1.1.6}
\end{eqnarray}
where $ r = r^1 \ot r^2 \ot r^3 = R$. Under the above
correspondence the braiding $c$ corresponding to $R$ is given by
the formula
\begin{equation}\eqlabel{1.1.1}
c_{M, N}: M \ot_A N \to N\ot_A M, \quad c_{M,N}(m\ot_A n) =
R^1nR^2\ot_A mR^3
\end{equation}
for all $M$, $N \in {}_A\Mm_A$, $m\in M$ and $n\in N$.\\
An invertible element $R \in A^{(3)}$ satisfying
\equref{1.1.2}-\equref{1.1.6} is called a canonical R-matrix and
$(A, R)$ is called an algebra with a canonical R-matrix.
\end{theorem}

\begin{proof} $A^{(2)}$ is an $A$-bimodule via the usual actions
$ a (x \ot y) b = ax \ot yb $, for all $a$, $b$, $x$, $y\in A$.
Let $c:\ {}_A\Mm_A\times  {}_A\Mm_A\to
 {}_A\Mm_A\times  {}_A\Mm_A$ be a braiding on ${}_A\Mm_A$. For each $M,N\in  {}_A\Mm_A$,
 we have an $A$-bimodule isomorphism $c_{M,N}:\ M\ot_A N\to N\ot_A M$, that is
 natural in $M$ and $N$. Now consider
 $$c_{A^{(2)},A^{(2)}}:\ A^{(3)}\cong A^{(2)}\ot_A A^{(2)}\to
 A^{(3)}\cong A^{(2)}\ot_A A^{(2)},$$
 and let $R=c_{A^{(2)},A^{(2)}}(1\ot 1\ot 1)$. $c$ is completely determined by $R$.
 For $M,N\in {}_A\Mm_A$, $m\in M$ and $n\in N$, we consider the
 $A$-bimodule maps $f_m:\ A^{(2)}\to M$ and $g_n:\ A^{(2)}\to N$ given by the
 formulas $f_m(a\ot b)=amb$ and $g_n(a\ot b)=anb$. From the naturality of $c$,
 it follows that
 $$(g_n\ot_A f_m)\circ c_{A^{(2)},A^{(2)}}= c_{M,N}\circ (f_m\ot_A g_n).$$
 \equref{1.1.1} follows after we evaluate this formula
at $1\ot 1\ot 1$.
Obviously $c_{M,N}(ma\ot_A n)=c_{M,N}(m\ot_A
an)$. Furthermore,
  $c_{M,N}(am\ot_A n)=ac_{M,N}(a\ot_A n)$ and $c_{M,N}(m\ot_A na)=c_{M,N}(a\ot_A n)a$
  since $c_{M,N}$ is a bimodule map. If we write these three formulas down in the case
  where $M=N=A^{(2)}$, and $m=n=1\ot 1$, then we obtain (\ref{eq:1.1.2}-\ref{eq:1.1.4}).
  $c$ satisfies the two triangle equalities
  \begin{eqnarray*}
  c_{M\ot_AN,P}&=&(c_{M,P}\ot_A N)\circ (M\ot_A c_{N,P});\\
  c_{M,N\ot_AP}&=&(N\ot_A c_{M,P})\circ (c_{M,N}\ot_AP).
  \end{eqnarray*}
  The first equality is equivalent to
  $$R^1pR^2\ot_A m\ot_A nR^3=r^1R^1pR^2r^2\ot_A mr^3\ot_A nR^3,$$
  for all $m\in M$, $n\in N$ and $p\in P$. If we take $M=N=P=A^{(2)}$
  and $m=n=p=1\ot 1$, then we find that
  $R^1\ot R^2\ot 1\ot R^3=r^1R^1\ot R^2r^2\ot r^3\ot R^3$.
  Applying \equref{1.1.4}, we find that \equref{1.1.5} holds. In a similar way, the second triangle
  equality implies \equref{1.1.6}.\\
  We can apply the same arguments to the inverse braiding $c^{-1}$.
We set $S=S^1\ot S^2\ot S^3=c^{-1}_{A^{(2)},A^{(2)}}(1\ot 1\ot
1)$. Then we have that
 $$
  m\ot_A n= (c^{-1}_{N,M}\circ c_{M,N})(m\ot_A n)
  = c^{-1}_{N,M}(R^1nR^2\ot_A mR^3)
= S^1mR^3S^2\ot_A R^1nR^2S^3.$$

Now take $m=n=1\ot 1\in A^{(2)}$. Then we find
  \begin{eqnarray*}
 && \hspace*{-15mm}
  1\ot1\ot1=
  S^1\ot R^3S^2R^1\ot R^2S^3
  \equal{\equref{1.1.3}}
   R^1S^1\ot R^3S^2\ot R^2S^3\\
 & \equal{\equref{1.1.4}}&
   R^1S^1\ot S^2\ot R^2S^3R^3
  \equal{\equref{1.1.2}}
   R^1S^1R^2\ot S^2\ot S^3R^3
   \equal{\equref{1.1.3}}
   S^1R^2\ot S^2R^1\ot S^3R^3.
   \end{eqnarray*}
   In a similar way, we have that $R^1S^2\ot R^2S^1\ot R^3S^3=1\ot 1\ot 1$, and it follows that
   $S^2\ot S^1\ot S^3$ is the inverse of $R^1\ot R^2\ot R^3$.\\
   Conversely, assume that $R\in A^{(3)}$ is invertible and satisfies (\ref{eq:1.1.2}-\ref{eq:1.1.6}).
   Then we define $c$ using \equref{1.1.1}. Straightforward computations show that $c$ is
   a braiding on ${}_A\Mm_A$.
 \end{proof}

 Let $c$ be a braiding on ${}_A\Mm_A$ and $R$ the corresponding canonical
$R$-matrix. Then $c$ is a symmetry, if and only if $S = R$, this
means that
\begin{equation}\eqlabel{2.1.1}
 R^2r^1\ot R^1r^2\ot R^3r^3=1\ot 1\ot 1
\end{equation}
that is, $R^{-1}=R^2\ot R^1\ot R^3$. The next theorem shows that
the list of equations satisfied by an $R$-matrix from \thref{1.1}
can be reduced to two equations and furthermore, we prove that all
braidings on the category of $A$-bimodules are symmetries.

\begin{theorem}\thlabel{2.1}
Let $A$ be a $k$-algebra. Then there is a bijection between the
set of canonical $R$-matrices and the set of all elements $R\in
A^{(3)}$ satisfying \equref{1.1.4} and the normalizing condition
\begin{equation}\eqlabel{2.1.2}
R^1R^2\ot R^3 = R^2\ot R^3R^1 = 1\ot 1
\end{equation}
Furthermore, $R$ is invariant under cyclic permutation of the
tensor factors,
\begin{equation}\eqlabel{2.1.3}
R=R^2\ot R^3\ot R^1=R^3\ot R^1\ot R^2,
\end{equation}
and we have the additional normalizing condition
\begin{equation}\eqlabel{2.1.4}
 R^1 \ot R^2 R^3 = 1\ot 1.
\end{equation}
In particular, every braiding on ${}_A\Mm_A$ is a symmetry.
 \end{theorem}

 \begin{proof}
Let $R$ be an $R$-matrix as in \thref{1.1}, i.e. $R$ is invertible
and satisfies (\ref{eq:1.1.2}-\ref{eq:1.1.6}). Multiplying the
second and the third tensor factor in \equref{1.1.6}, we find that
$R=R^1\ot R^2r^1 r^2\ot R^3r^3=R(1\ot r^1r^2\ot r^3)$. From the
fact that $R$ is invertible, it follows that $1\ot 1\ot 1=1\ot
r^1r^2\ot r^3$, and the first normalizing condition of
\equref{2.1.2} follows after we multiply the first two tensor
factors. On the other hand, if we apply the flip map on the last
two positions in \equref{1.1.5} we obtain that $R^1\ot R^2\ot R^3
\ot 1 = r^1R^1\ot r^2\ot R^3 \ot r^3R^2$. Multiplying the last two
positions we obtain:
\begin{eqnarray*}
R = r^1R^1\ot r^2\ot R^3 r^3R^2 \, \equal{\equref{1.1.2}} \, r^1
R^3 R^1 \ot r^2 \ot r^3 R^2 = R (R^3 R^1 \ot 1 \ot R^2)
\end{eqnarray*}
As $R$ is invertible it follows that $R^3 R^1 \ot R^2 = 1\ot 1$,
as needed.\\
Conversely, assume now that $R$ satisfies \equref{1.1.4} and
\equref{2.1.2}. We will show that $R$ is a canonical $R$-matrix satisfying
\equref{2.1.1} and hence from the observations preceding
\thref{2.1} we obtain that the braiding $c$ corresponding to $R$
is a symmetry. First we show that $R$ is invariant under cyclic
permutation of the tensor factors.
 \begin{eqnarray*}
 &&\hspace*{-2cm}
 R^3\ot R^1\ot R^2
 \, \equal{\equref{2.1.2}} \, R^3r^1r^2\ot r^3R^1\ot R^2
 \, \equal{\equref{1.1.4}} \, R^3r^2\ot r^3R^1\ot r^1R^2\\
  &\equal{\equref{1.1.4}}&r^2\ot r^3R^3R^1\ot r^1R^2
   \, \equal{\equref{2.1.2}} \, r^2\ot r^3\ot r^1.
   \end{eqnarray*}
This implies immediately that the central conditions
(\ref{eq:1.1.2}-\ref{eq:1.1.3}) are also satisfied. Next we show
that (\ref{eq:1.1.5}-\ref{eq:1.1.6}) are satisfied.
 \begin{eqnarray*}
 &&\hspace*{-2cm}
 r^1R^1\ot r^2\ot r^3R^2\ot R^3
 \, \equal{\equref{1.1.3}} \,
R^1\ot r^2\ot r^3R^2r^1\ot R^3\\
& \equal{\equref{1.1.4}}&R^1\ot R^2r^2\ot r^3r^1\ot R^3
\, \equal{\equref{2.1.2}} \, R^1\ot R^2\ot 1\ot R^3;\\
&&\hspace*{-2cm} R^1\ot R^2r^1\ot r^2\ot R^3r^3 \,
\equal{\equref{1.1.4}} \,
R^1\ot r^3R^2r^1\ot r^2\ot R^3\\
&\equal{\equref{1.1.4}}& R^1\ot r^3r^1\ot R^2r^2\ot R^3 \,
\equal{\equref{2.1.2}} \, R^1\ot 1\ot R^2\ot R^3.
\end{eqnarray*}
Finally, we prove that \equref{2.1.1} holds since
 \begin{eqnarray*}
 &&\hspace*{-2cm}
 R^1r^2\ot R^2r^1\ot R^3r^3
  \, \equal{\equref{1.1.3}} \,
R^1r^2R^2\ot r^1\ot R^3r^3 \, \equal{\equref{1.1.4}} \,
R^1R^2\ot r^1\ot R^3r^2r^3\\
& \equal{\equref{2.1.2}}& 1\ot r^1\ot r^2r^3 \,
\equal{\equref{2.1.3}} \, 1\ot r^3\ot r^1r^2 \,
\equal{\equref{2.1.2}} \, 1\ot 1\ot 1.
\end{eqnarray*}
and the proof is finished.
\end{proof}

The commutative case is completely classified by the following result.

\begin{proposition}\prlabel{1.2} Let $A$ be a $k$-algebra. Then:
\begin{enumerate}
\item If a monomial $x\ot y\ot z$ is a canonical $R$-matrix, then it is
equal to $1\ot 1\ot 1$.
\item $1\ot 1\ot 1$ is a canonical $R$-matrix if and only if $u_A:\ k\to A$
is an epimorphism of
 rings.
\item If $A$ is commutative, then $(A, R)$ is an algebra with a
canonical R-matrix if and only if $R = 1 \ot 1 \ot 1$ and $u_A : k
\to A$ is an epimorphism of rings.
\end{enumerate}
\end{proposition}

 \begin{proof}
 1. Let $R = x\ot y\ot z$ be a canonical $R$-matrix. From (\ref{eq:1.1.5}-\ref{eq:1.1.6}), it follows that
 $$x \ot 1 \ot y \ot z = x \ot yx \ot y \ot z^{2}~~{\rm and}~~x \ot y \ot 1 \ot z = x^{2} \ot y \ot zy \ot z.$$
 Since $R$ is invertible, this implies that
 $$1\ot 1\ot 1\ot 1=1\ot yx\ot 1\ot z~~{\rm and}~~1\ot 1\ot 1\ot 1=x\ot 1\ot zy\ot 1,$$
 and, multiplying tensor factors, we find that
 $1\ot 1=yx\ot z~~{\rm and}~~1\ot 1=x\ot zy$.
 It then follows that $yxz=xzy=1$, hence $y$ is invertible with $y^{-1}=xz$. Finally
 $$x\ot y\ot z=x\ot y^2y^{-1}\ot z\equal{\equref{1.1.4}}x\ot yy^{-1}\ot z y = 1\ot 1\ot 1.$$
 2. If $R=1\ot 1\ot 1$, then the three centralizing conditions (\ref{eq:1.1.2}-\ref{eq:1.1.4})
 are equivalent to $a\ot 1=1\ot a$, for all $a\in A$, which is equivalent to $u_A:\ k\to A$
 being an epimorphism of rings, see \cite{silver}.\\
3. Assume that $(A, R)$ is an algebra with a canonical R-matrix. Then:
 \begin{eqnarray*}
 &&\hspace*{-2cm}
R^1\ot R^2\ot 1\ot R^3 \, \equal{\equref{1.1.5} } \, r^1 R^1 \ot r^2 \ot r^3 R^2 \ot R^3\\
&\equal{\equref{1.1.4}}&  r^1 R^1 \ot R^2 r^2  \ot r^3 \ot R^3
= \sum R^1 r^1 \ot R^2 r^2  \ot r^3 \ot R^3.
\end{eqnarray*}
At the third step, we used the fact that $A$ is commutative.
From the fact that $R$ is invertible, it follows that $R^1\ot R^2\ot 1\ot R^3=1 \ot 1 \ot 1\ot 1$
and $R=1 \ot 1 \ot 1$. The rest of
the proof follows from 2.
\end{proof}

\begin{remarks}\relabel{1.1.b}
1. The notion of quasi-triangular bialgebroid was introduced in \cite[Def. 19]{Donin}. Quasi-triangular
structures on a bialgebroid are given by universal R-matrices, see \cite[Prop. 20]{Donin} and
\cite[Def. 3.15]{Bohm},
and correspond bijectively to braidings on the category of modules over the bialgebroid
\cite[Theorem 3.16]{Bohm}. It is well-known that $A^e$ is an $A$-bialgebroid, with the Sweedler canonical
coring as underlying coring, and $A$-bimodules are left $A^e$-modules.
Comparing our \thref{1.1} with the (left handed) version of  \cite[Theorem 3.16]{Bohm} yields that canonical
R-matrices for $A$ correspond bijectively to universal R-matrices for the canonical bialgebroid
$A^e$. This leads to an alternative proof of \thref{1.1}, if we identify the R-matrices from \cite{Donin}
with our R-matrices. This, however, is more complicated than the straightforward proof that we
presented, that also has the advantage that it is self-contained and avoids all technicalities
on bialgebroids.

2. (\ref{eq:1.1.5}-\ref{eq:1.1.6}) can be rewritten as $R^{124} =
R^{123}R^{134}$ and $R^{134} = R^{124} R^{234}$ in the algebra
$A^{(4)}$.

3. It follows from \prref{1.2} that there is only one braiding on the category of (left)
$k$-modules, namely the one given by the usual switch map.
\end{remarks}

Before we state our next main result \thref{2.2}, we need a technical Lemma.
If $M\in {}_A\Mm_A$, then $A\ot M$ is a $k\ot A$-bimodule, and we can consider
$$(A\ot M)^{k\ot A}=\{\sum_i a_i\ot m_i\in A\ot M~|~\sum_i a_i\ot am_i=\sum_i a_i\ot m_ia,~~{\rm for~all}~
a\in A\}.$$ If $M=A^{(2)}$, then $(A\ot A^{(2)})^{k\ot A}$ is the
set of elements $R\in A^{(3)}$ satisfying \equref{1.1.4}. We have
a map $\alpha_M:\ A\ot M^A\to (A\ot M)^A$, $\alpha_M(a\ot m) =a\ot
m$.

\begin{lemma}\lelabel{free}
Let $M$ be an $A$-bimodule. The map $\alpha_M$ is injective if $A$
is flat as a $k$-module, and bijective if $A$ is free as a
$k$-module.
\end{lemma}

\begin{proof}
If $A$ is flat, then $A\ot M^A\to A\ot M$ is injective, and then $\alpha_M$ is also injective.\\
Assume that $A$ is free as a $k$-module, and let $\{e_j~|~j\in
I\}$ be a free basis of $A$. Assume that $x=\sum_i a_i\ot m_i\in
(A\ot M)^{k\ot A}$. For all $i$, we can write $a_i=\sum_{j\in I}
\alpha^j_ie_j$, for some $\alpha^j_i\in k$. Then $x=\sum_{j\in I}
e_j\ot \bigl(\sum_i\alpha^j_im_i\bigr)$. Now
$$x=\sum_{j\in I} e_j\ot \bigl(\sum_i\alpha^j_iam_i\bigr)=
\sum_i a_i\ot am_i=\sum_i a_i\ot m_ia= \sum_{j\in I} e_j\ot
\bigl(\sum_i\alpha^j_im_ia\bigr),$$ hence
$\sum_i\alpha^j_iam_i=\sum_i\alpha^j_im_ia$, for all $j\in I$, and
$\sum_i\alpha^j_im_i\in M^A$. We conclude that $x=\sum_{j\in I}
e_j\ot \bigl(\sum_i\alpha^j_im_i\bigr)\in \im \alpha_M$, and this
shows that $\alpha_M$ is surjective.
\end{proof}

In our next result we assume $A$ to be flat over $k$. Hence the
map $\alpha_{A^{(2)}}$ is injective and this will allows us to
identify the elements in $A \ot (A \ot A)^{A}$ with the elements
in $A \ot A \ot A$ satisfying \equref{1.1.4}.

\begin{theorem}\thlabel{2.2}
Let $A$ be a flat $k$-algebra and consider the conditions:
\begin{enumerate}
\item $(F,G)$ is a pair of inverse equivalences, that is, $A$ is
an Azumaya algebra; \item The functor $G=(-)^A:\ {}_A\Mm_A\to
\Mm_k$ is fully faithful; \item the functor $G=(-)^A:\
{}_A\Mm_A\to \Mm_k$ is separable; \item there exists $R=R^1\ot
R^2\ot R^3\in A\ot (A\ot A)^A$ such that $R^1R^2\ot R^3=1\ot 1$;
\item there exists a unique $R=R^1\ot R^2\ot R^3\in A\ot (A\ot
A)^A$ such that $R^1R^2\ot R^3=1\ot 1$; \item there exists a
braiding on ${}_A\Mm_A$, that is, there exists $R\in A^{(3)}$ such
that $(A,R)$ is an algebra with a canonical R-matrix.
\end{enumerate}
Then $(1)\Rightarrow (2)\Leftrightarrow (3)\Leftrightarrow (4)\Leftrightarrow (5) \Rightarrow (6)$.
If $A$ is central, then $(2)\Rightarrow (1)$. If $A$ is free as a $k$-module, then $(6)\Rightarrow (5)$,
and in this case the braiding on ${}_A\Mm_A$ is unique. If $k$ is a field, and $A$ is finite dimensional,
then $(6)\Rightarrow (1)$, and all six assertions are equivalent.
\end{theorem}

\begin{proof}
$\ul{(1) \Rightarrow (2)}$, $(2)\Rightarrow (3)$ and $(5) \Rightarrow (4)$ are trivial.\\
$\ul{(3) \Rightarrow (4)}$. If $G$ is separable, then we have a
natural transformation $\zeta:\ 1\Rightarrow FG$ such that
$\varepsilon_M\circ \zeta_M=M$, for all $M\in {}_A\Mm_A$. Now let
$R=\zeta_{A^{(2)}}(1\ot 1)=R^1\ot R^2\ot R^3\in FG(A^{(2)})=A\ot
(A\ot A)^A$.
Then $1\ot 1=(\varepsilon_{A^{(2)}}\circ \zeta_{A^{(2)}})(1\ot 1)= R^1R^2\ot R^3$.\\
The natural transformation $\zeta$ is completely determined by $R$. For an $A$-bimodule $M$
and $m\in M$, we define $f_m$ as in the proof of \thref{1.1}. From the naturality of $\zeta$, it
follows that the diagram
$$\xymatrix{
A^{(2)}\ar[d]_{f_m}\ar[rr]^{\zeta_{A^{(2)}}}&& A\ot (A\ot A)^A\ar[d]^{A\ot (f_m)^A}\\
M\ar[rr]^{\zeta_M}&&A\ot M^A}$$
commutes. Evaluating the diagram at $1\ot 1$, we find that
\begin{equation}\eqlabel{2.2.1}
\zeta_M(m)=R^1\ot R^2mR^3.
\end{equation}
$\ul{(4) \Rightarrow (6)}$. Write $R=\sum_i a_i\ot b_i$, with
$a_i\in A$ and $b_i\in (A\ot A)^A$. Then $R^2\ot R^3R^1=\sum_i
b_ia_i=\sum_i a_ib_i=R^1R^2\ot R^3=1\ot 1$, thus $R$ satisfies
\equref{2.1.2}. Moreover, as $R\in A \ot (A\ot A)^{A}$ it follows
that $R$ also satisfies \equref{1.1.4}. Then using \thref{2.1} we
obtain that $R$ is a canonical $R$-matrix and it determines a
braiding on ${}_A\Mm_A$.

$\ul{(4) \Rightarrow (2)}$. Given $R\in A\ot (A\ot A)^A$
satisfying $R^1R^2\ot R^3=1\ot 1$, we define $\zeta$ using
\equref{2.2.1}. It follows immediately that $(\varepsilon_M\circ
\zeta_M)(m)=
\varepsilon(R^1\ot R^2mR^3)=R^1R^2mR^3=m$.\\
We have seen in the proof of $\ul{4) \Rightarrow 6)}$ that
\equref{1.1.3} and \equref{2.1.4} are satisfied. For $a_i\in A$ and $m_i\in M^A$, we then compute
\begin{eqnarray*}
&&\hspace*{-2cm} (\zeta_M \circ \varepsilon_M ) (\sum_i a_i \ot
m_i) = \sum_i R^1 \ot R^2 a_i m_i R^3 \, \equal{\equref{1.1.3}} \,
\sum_i a_i R^1
\ot R^2 m_i R^3 \\
&= & \sum_i a_i R^1 \ot R^2 R^3 m_i \, \equal{\equref{2.1.4}} \,
\sum_i a_i \ot m_i.
\end{eqnarray*}
This shows that $\varepsilon$ is a natural transformation with inverse $\zeta$, and $G$ is fully faithful.\\
$\ul{(2) \Rightarrow (5)}$. We have already seen that 2) implies 4), and this shows that $R$ exists.
If $G$ is fully faithful, then $\varepsilon_M$ is invertible, for all $M\in {}_A\Mm_A$. If $R\in A\ot (A\ot A)^A$
satisfies $R^1R^2\ot R^3=1\ot 1$, then $\varepsilon_{A\ot A}(R)=1\ot 1$, hence $R=
\varepsilon_{A\ot A}^{-1}(1\ot 1)$.\\
$\ul{(6) \Rightarrow (4)}$. From (5), it follows that there exists
$R\in (A\ot A^{(2)})^{k\ot A}$ such that $R^1R^2\ot R^3=1\ot 1$,
see \thref{2.1}. $\alpha_{A^{(2)}}$ is bijective, see
\leref{free}, hence $\alpha_{A^{(2)}}^{-1}(R)\in
A\ot (A\ot A)^A$ satisfies (3). The uniqueness of $R$ follows from (4).\\
$\ul{(4)\Rightarrow (1)}$. Assume that $A$ is central. From (4),
it follows that $\varepsilon_{A\ot A}: A\ot (A\ot A)^A\to A\ot A$
is surjective, and then it follows from \cite[Theorem 3.1]{AG}
that $A$
is separable over $Z(A)=k$. Thus $A$ is central separable, and therefore Azumaya.\\
$(6)\Rightarrow (1)$. If $k$ is a field, then $A$ is free, so $(6)$ implies $(5)$, and, a fortiori, $(2)$.
Then $\varepsilon_A : A \ot A^A \to A$ is an isomorphism of $A$-bimodules, and therefore also
of vector spaces. A count of dimensions shows that ${\rm dim}_k (Z (A)) = {\rm dim}_k (A^A)
= 1$, so that $Z(A) = k 1_A$, and $A$ is central, and then $(1)$ follows from $(2)$.
\end{proof}

In particular, applying \thref{2.2} for finite dimensional
algebras over fields we obtain the following surprising
characterization of central simple algebras:

\begin{corollary}\colabel{centralsimple}
For a finite dimensional algebra $A$ over a field $k$, the
following assertions are equivalent:
\begin{enumerate}
\item $A$ is a central simple algebra;
\item there exists a (unique) braiding on ${}_A\Mm_A$;
\item there exists a (unique) invertible element $R \in A \ot A
\ot A$ satisfying the conditions
$
R^1\ot a \, R^2\ot R^3 = R^1\ot R^2\ot R^3 \, a$ and $R^1R^2\ot
R^3 = R^2\ot R^3R^1 = 1\ot 1
$,
for all $a\in A$.
\end{enumerate}
\end{corollary}

For any $k$-algebra $A$, the functor $F:\ \Mm_k\to {}_A\Mm_A$ is strong monoidal. Indeed,
for any $N,N'\in \Mm_k$, we have natural isomorphisms $\varphi_0:\ F(k)=A\ot k\to A$ and
$$\varphi_{N,N'}:\ F(N)\ot_A F(N')=(A\ot N)\ot_A (A\ot N')\to F(N\ot N')=A\ot N\ot N'$$
satisfying all the necessary axioms, see \cite{K}.

\begin{proposition}\prlabel{2.4}
Let $(A, R)$ be an algebra with a canonical R-matrix. Then the
functor $F:\ \Mm_k\to {}_A\Mm_A$ preserves the symmetry.
\end{proposition}

\begin{proof}
We have to show that the following diagram commutes
$$\xymatrix{(A\ot N)\ot_A(A\ot N')\ar[d]_{c_{A\ot N,A\ot N'}}\ar[rr]^{\varphi_{N,N'}}&&
A\ot N\ot N'\ar[d]^{A\ot\tau_{N,N'}}\\
(A\ot N')\ot_A(A\ot N)\ar[rr]^{\varphi_{N',N}}&&A\ot N'\ot N}$$
Here $\tau_{N,N'}:\ N\ot N'\to N'\ot N$ is the usual switch map.
For $a, b\in A$, $n\in N$ and $n'\in N'$, we compute
\begin{eqnarray*}
&&\hspace*{-2cm} (\varphi_{N',N}\circ c_{A\ot N,A\ot N'})((a\ot
n)\ot_A(b\ot n')) \, \equal{\equref{1.1.1}} \, \varphi_{N',N}\bigl((R^1bR^2\ot n')\ot_A aR^3\ot n\bigr)\\
&=& R^1bR^2 aR^3\ot n'\ot n \, \equal{\equref{1.1.2}} \,
R^1R^2 abR^3\ot n'\ot n\\
&\equal{\equref{2.1.2}}& ab\ot n'\ot n = ab \ot \tau_{N,N'}(n\ot n')\\
&=&((A\ot \tau_{N,N'})\circ \varphi_{N,N'})((a\ot n)\ot_A(b\ot
n'))
\end{eqnarray*}
and the proof is finished.
\end{proof}

If $A$ is an Azumaya algebra, then it follows from \thref{2.2}
that we have a symmetry on the category of $A$-bimodules
${}_A\Mm_A$. In Examples \ref{ex:2.5} and  \ref{ex:2.6}, we give
explicit formulas for $R$ in the case where $A$ is a matrix ring
or a quaternion algebra; in both cases $A$ is free, so that the
canonical $R$-matrix is unique.

\begin{example}\exlabel{2.5}
Let $A = M_n(k)$ be a matrix algebra. Then the $R$-matrix for
$M_n(k)$ if given by
$$
R = \sum_{i,j,k=1}^n e_{ij}\ot e_{ki}\ot e_{jk}
$$
where $e_{ij}$ is the elementary matrix with $1$ in the
$(i,j)$-position and $0$ elsewhere. Indeed, for all indices $i$,
$j$, $p$, $q$, we have
$$e_{pq}(\sum_{k=1}^n e_{ki}\ot e_{jk})=e_{pi}\ot e_{jq}= (\sum_{k=1}^n e_{ki}\ot e_{jk})e_{pq},$$
hence $\sum_{k=1}^n e_{ki}\ot e_{jk}\in (A\ot A)^A$ and $R =
\sum_{i,j=1}^n e_{ij}\ot(\sum_{k=1}^n e_{ki}\ot e_{jk}) \in A\ot
(A\ot A)^A$. Finally
$$
\sum_{i,j,k=1}^n e_{ij} e_{ki}\ot e_{jk}=\sum_{i,j=1^n} e_{ii}\ot
e_{jj}=1\ot 1
$$
as needed.
\end{example}

\begin{example}\exlabel{2.6}
Let $K$ be a commutative ring, such that $2$ is invertible in $K$,
and take two invertible elements $a,b\in K$. The generalized
quaternion algebra $A={}^aK^b$ is the free $K$-module with basis
$\{1,i,j,k\}$ and multiplication defined by
$$i^{2} = a  , \qquad j^{2} = b, \qquad ij = -ji = k.$$
It is well-known that $A$ is an Azumaya algebra. The corresponding $R$-matrix is
\begin{eqnarray*}
&&\hspace*{-10mm}
R={1\over 4}(1\ot 1\ot 1)
+ {1\over 4a}(1\ot i\ot i+i\ot 1\ot i +i\ot i\ot 1)\\
&+& {1\over 4b}(1\ot j\ot j+j\ot 1\ot j +j\ot j\ot 1)
- {1\over 4ab}(1\ot k\ot k+k\ot 1\ot k +k\ot k\ot 1)\\
&+& {1\over 4ab}(i\ot j\ot k +j\ot k\ot i + k\ot i\ot j)
- {1\over 4ab}(j\ot i\ot k +k\ot j\ot i + i\ot k\ot j).
\end{eqnarray*}
It is easy to show that $R$ satisfies \equref{1.1.4} and \equref{2.1.2}. Indeed,
\begin{eqnarray*}
&&\hspace*{-10mm}
R^1R^2\ot R^3=
{1\over 4}(1\ot 1)
+ {1\over 4a}(i\ot i+i\ot i +a\ot 1)\\
&+& {1\over 4b}( j\ot j+j\ot j +b\ot 1)
- {1\over 4ab}(k\ot k+k\ot k -ab\ot 1)\\
&+& {1\over 4ab}(k\ot k -bi\ot i -a j\ot j)
+ {1\over 4ab}(k\ot k -bi\ot i -aj\ot j)=1\ot 1,
\end{eqnarray*}
proving the first normalization from \equref{2.1.2}; the second
one follows in a similar manner. An elementary computation shows
that $R^1\ot xR^2\ot R^3=R^1\ot R^2\ot R^3x$, for $x=i,j,k$, and
this proves \equref{1.1.4}.
\end{example}

\begin{example}\exlabel{2.7}
Any Azumaya algebra $A$ admits a canonical R-matrix. The converse is not
true: it suffices to consider $\QQ$ as a $\ZZ$-algebra. Since $\ZZ
\subset \QQ$ is an epimorphism of rings, it follows from
\prref{1.2} that $(\QQ, \, 1\ot 1\ot 1)$ is braided and it is
obvious that $\QQ$ is not a $\ZZ$-Azumaya algebra.
\end{example}

\begin{example}\exlabel{2.9}
Let $(A, R)$, $(B, S)$ be two algebras with a canonical R-matrix. It is
straightforward to show that $(A\ot B, T)$, with $T := R^1 \ot S^1
\ot R^2 \ot S^2 \ot R^3 \ot S^3 \in (A \ot B)^{(3)}$, is an algebra with a
canonical R-matrix.
\end{example}

We conclude this paper with another application of canonical R-matrices:
they can be applied to deform the switch map into a simultaneous solution
of the quantum Yang-Baxter equation and the braid equation.

\begin{theorem}\thlabel{qybe}
Let $(A, R)$ be an algebra with a canonical R-matrix and $V$ an
$A$-bimodule. Then the map
$$
\Omega:  V \ot V \rightarrow V \ot V, \qquad  \Omega (v \ot w) =
R^{1} \, w R^{2} \ot R^{3} \, v
$$
is a solution of the quantum Yang-Baxter equation $\Omega^{12} \,
\Omega^{13} \, \Omega^{23} = \Omega^{23} \, \Omega^{13} \,
\Omega^{12}$ and of the braid equation $\Omega^{12} \, \Omega^{23}
\, \Omega^{12} = \Omega^{23} \, \Omega^{12} \, \Omega^{23}$.
Moreover $\Omega^3 = \Omega$ in $\End (V^{(2)})$.
\end{theorem}

\begin{proof} Let $r = S = R$. Then for any $v$, $w$, $t \in V$ we have:
\begin{eqnarray*}
&&\hspace*{-15mm} \Omega^{12} \, \Omega^{13} \, \Omega^{23}(v \ot
w \ot t) = \Omega^{12} \, \Omega^{13}(v \ot R^{1} \, t \, R^{2}
\ot R^{3} \,
w)\\
&{=}& \Omega^{12}(r^{1} \, R^{3} \, w \, r^{2} \ot R^{1} \, t \,
R^{2} \ot r^{3} \, v) = S^{1} \, R^{1} \, t \, R^{2} \, S^{2} \ot
S^{3} \, r^{1} \,
R^{3} \, w \, r^{2} \ot r^{3} \, v\\
&\equal{\equref{1.1.3}}& R^{1} \, t \, R^{2} \, S^{1} \, S^{2} \ot
S^{3} \, r^{1} \, R^{3} \, w \, r^{2} \ot r^{3} \, v \,
\equal{\equref{2.1.2}} \, R^{1} \, t \, R^{2} \ot
r^{1} \, R^{3} \, w \, r^{2} \ot r^{3} \, v\\
&\equal{\equref{1.1.2}}& R^{1} \, t \, R^{2} \ot r^{1} \, w \,
r^{2} \ot R^{3} \, r^{3} \, v
\end{eqnarray*}
and
\begin{eqnarray*}
&&\hspace*{-15mm} \Omega^{23} \, \Omega^{13} \, \Omega^{12}(v \ot
w \ot t) = \Omega^{23} \, \Omega^{13} (R^{1} \, w \, R^{2} \ot
R^{3} \, v \ot
t)\\
&{=}& \Omega^{23} (r^{1} \, t \, r^{2} \ot R^{3} \, v \ot r^{3} \,
R^{1} \, w \, R)^{2}) = r^{1} \, t \, r^{2} \ot S^{1} \, r^{3} \,
R^{1} \, w \, R^{2} \,
S^{2} \ot S^{3} \,  R^{3} \, v\\
&\equal{\equref{1.1.2}}& r^{1} \, t \, r^{2} \ot S^{1} \, R^{1} \,
w \, R^{2} \, S^{2} \ot r^{3} \, S^{3} \,  R^{3} \, v
\, \equal{\equref{1.1.3}} \, r^{1} \, t \, r^{2} \ot R^{1} \, w \, R^{2} \, S^{1} \, S^{2} \ot r^{3} \, S^{3} \,  R^{3} \, v\\
&\equal{\equref{2.1.2}}& r^{1} \, t \, r^{2} \ot R^{1} \, w \,
R^{2} \ot r^{3} \,  R^{3} \, v,
\end{eqnarray*}
Hence $\Omega$ is a solution of the quantum Yang-Baxter equation.
On the other hand:
\begin{eqnarray*}
&&\hspace*{-15mm} \Omega^{12} \, \Omega^{23} \, \Omega^{12}(v \ot
w \ot t) =
\Omega^{12} \, \Omega^{23} (R^1 \, w \, R^2 \ot R^{3} \, v \ot t )\\
&{=}& \Omega^{12} (R^1 \, w \, R^2 \ot r^1 \, t \, r^2 \ot r^{3}
\, R^3 v) = S^{1} \, r^{1} \, t \, r^{2} \, S^{2} \ot S^{3} \,
R^{1} \, w \, R^2 \ot r^{3} \, R^3 \, v\\
&\equal{\equref{1.1.3}}& r^{1} \, t \, r^{2} \, S^{1} \,  S^{2}
\ot S^{3} \, R^{1} \, w \, R^2 \ot r^{3} \, R^3 \, v \,
\equal{\equref{2.1.2}} \,  r^{1} \, t \, r^{2} \ot R^{1} \, w \,
R^2 \ot r^{3} \, R^3 \, v
\end{eqnarray*}
and
\begin{eqnarray*}
&&\hspace*{-15mm} \Omega^{23} \, \Omega^{12} \, \Omega^{23} (v \ot
w \ot t) = \Omega^{23} \, \Omega^{12} ( v\ot r^{1} \, t \, r^{2}
\ot r^{3} \, w )\\
&{=}& \Omega^{23} (S^{1} \, r^1 \, t \, r^{2} \, S^2 \ot S^{3} \,
v \ot r^{3} \, w ) = S^{1} \, r^1 \, t \, r^{2} \, S^2 \ot R^{1}
\, r^{3} \, w\,  R^{2} \ot R^3 \, S^3 v\\
&\equal{\equref{1.1.3}}& r^1 \, t \, r^{2} \, S^{1} \, S^2 \ot
R^{1} \, r^{3} \, w\,  R^{2} \ot R^3 \, S^3 v \,
\equal{\equref{2.1.2}} \, r^1 \, t \, r^{2} \ot R^{1} \, r^{3} \,
w\,
R^{2} \ot R^3 \, v \\
&\equal{\equref{1.1.2}}& r^1 \, t \, r^{2} \ot R^{1} \, w\, R^{2}
\ot  r^{3} \, R^3 \,  v.
\end{eqnarray*}
Thus $\Omega$ is also a solution of the braid equation. Finally,
\begin{eqnarray*}
&&\hspace*{-15mm} \Omega^{3} (v \ot w) = S^1 \, r^3 \, R^1 \, w \,
R^2 \, S^2 \ot S^3
\, r^1 \, R^3 \, v \, r^2 \\
&\equal{\equref{1.1.2}}& S^1 \, w \, R^2 \, S^2 \ot r^3 \, R^1 \,
S^3 \, r^1 \, R^3 \, v \, r^2 \, \equal{\equref{2.1.4}} \, S^1 \,
w \, S^2 \ot r^3 \, R^1 \, S^3 \, R^2 \, r^1 \, R^3 \, v \, r^2 \\
&\equal{\equref{1.1.3}}& S^1 \, w \, S^2 \ot R^1 \, S^3 \, R^2 \,
r^3 \, r^1 \, R^3 \, v \, r^2 \, \equal{\equref{2.1.2}} \, S^1 \,
w \, S^2 \ot R^1 \, S^3 \, R^2 \, R^3 \, v \\
&\equal{\equref{2.1.3}}& S^1 \, w \, S^2 \ot  S^3 \, v \, \, = \,
\, \Omega (v \ot w).
\end{eqnarray*}
\end{proof}

\end{document}